\documentclass[10pt]{article}
\usepackage{amsmath,amssymb,graphicx,amsthm}
\usepackage{tikz}
\usetikzlibrary{matrix, arrows}

\newcommand{\charg}[3]{{{\rm H}}^{#1}({#2},{#3})}
\newtheorem{theorem}{Theorem}[section]
\newtheorem{lemma}[theorem]{Lemma}
\newtheorem{corollary}[theorem]{Corollary}
\newtheorem{proposition}[theorem]{Proposition}

\newtheorem{definition}[theorem]{Definition}

\begin{document}

\title{Universal deformation rings and fusion}

\author{David C. Meyer\\Department of Mathematics\\University of Iowa\\
Iowa City, IA 52242-1419\\ \textsf{david-c-meyer@uiowa.edu}}

\date{}

\maketitle

\begin{abstract}
We study in this paper the extent to which one can detect fusion in certain finite groups $\Gamma$ from information about the universal deformation rings $R(\Gamma,V)$ of absolutely irreducible $\mathbb{F}_p\Gamma$-modules $V$.  The $\Gamma$ we consider are extensions of either abelian or dihedral groups $G$ of order prime to $p$ by an elementary abelian $p$-group of rank 2.
\end{abstract}

\section{Introduction}
\label{s:intro}

This paper has to do with determining information about the internal structure of a finite group $\Gamma$ from the knowledge of the universal deformation rings $R(\Gamma,V)$ associated to absolutely irreducible $\mathbb{F}_p\Gamma$-modules $V$.  The kind of internal structure we will consider is the fusion of certain subgroups $N$ in $\Gamma$.  A pair of elements of $N$ are said to be fused in $\Gamma$ if they are conjugate in $\Gamma$, but not in $N$.  By determining the fusion of $N$ in $\Gamma$, we mean listing all such pairs.  The universal deformation ring $R(\Gamma,V)$ is characterized by the property that the isomorphism class of every lift of $V$ over a complete local commutative Noetherian ring $R$ with residue field $\mathbb{F}_p$ arises
from a unique local ring homomorphism $\alpha: R(\Gamma,V)\to R$.  Our main goal is thus to determine how to transfer information about the universal deformation rings to information about the structure of groups.  It is natural to expect a connection with fusion because fusion plays a key role in the character theory of $\Gamma$, which in turn enters into finding universal deformation rings of representations.

In this paper, we consider $\Gamma$ which are extensions of a group $G$ whose order is
relatively prime to $p$ by an elementary abelian $p$-group $N$ of rank 2. We can now state our main result:

\begin{theorem}
Let $G$ be a dihedral group of order $2n \geq 6$ and let $p$ be an odd prime such that $p \equiv 1$ mod $n$.  Fix an irreducible action of $G$ on $N = \mathbb{Z}/p\mathbb{Z} \times \mathbb{Z}/p\mathbb{Z}$, and let $\Gamma$ be the resulting semi-direct product of $G$ with $N$.  
\begin{enumerate}
\item[a.] If the center of $G$ acts trivially on $N$, then one can determine the fusion of $N$ in $\Gamma$ from the absolutely irreducible $\mathbb{F}_p\Gamma$-modules $V$ of dimension 2 over $\mathbb{F}_p$ which have universal deformation ring $R(\Gamma,V)$ different from $\mathbb{Z}_p$.
\item[b.] If the center of $G$ acts non-trivially on $N$, then $n$ is even and $R(\Gamma,V) \cong \mathbb{Z}_p$ for all absolutely irreducible $\mathbb{F}_p\Gamma$-modules $V$ of dimension 2 over $\mathbb{F}_p$.  In this case, one can determine the fusion of $N$ in $\Gamma$ if and only if $n$ is either a power of $2$, or $n = 2 q$, for some odd prime $q$.
 
\end{enumerate}
\end{theorem}

In section \ref{ss:ab} we prove a weaker result when $G$ is abelian.  In the course of proving Theorem 1.1, we must calculate $\charg{i}{\Gamma}{ \mathrm{Hom}_{\,\mathbb{F}_p}(V,V)}$, for $i = 1, 2$, since these enter into the computation of $R(\Gamma,V)$.

The paper is organized as follows. In section \ref{s:prelim}, we recall the definitions of deformations and
deformation rings, including some basic results. 
In section 3, we concentrate on the case when $\Gamma$ 
is an extension of a finite group $G$ by an elementary abelian $p$-group of rank $\ell \geq 2$. 
We give an explicit formula for the cohomology groups 
$\mathrm{H}^i(\Gamma,\mathrm{Hom}_{\,\mathbb{F}_p}(V,V))$ for $i=1,2$ for all projective
$\mathbb{F}_pG$-modules $V$ which are viewed as $\mathbb{F}_p\Gamma$-modules by inflation
(see Theorem 3.1).  In section 4, we prove our main results, Theorems \ref{th:no2} and \ref{th:no3}, on the connection between fusion and universal deformation rings, respectively cohomology groups, in the case when $G$ is a dihedral group.  In section 4.6, we briefly discuss the case when $G$ is abelian and compare this case to the dihedral one.

This paper is part of my dissertation at the University of Iowa under the supervision of Professor Frauke Bleher \cite{meyer}.  I would like to thank her for all of her advice and guidance.

\section{Preliminaries}
\label{s:prelim}

In this section, we give a brief introduction to universal deformation rings and deformations. 
For more background material, we refer the reader to \cite{mazur} and \cite{desmit-lenstra}.

Let  $p$ be an odd prime, $\mathbb{F}_p$ be the field with $p$ elements, and $\mathbb{Z}_p$ denote the ring of $p$-adic integers. Let 
$\hat{\mathcal{C}}$ be the category of all complete local commutative Noetherian rings with residue field 
$\mathbb{F}_p$. Note that all rings in $\hat{\mathcal{C}}$ have a natural $\mathbb{Z}_p$-algebra structure.
The morphisms in $\hat{\mathcal{C}}$ are continuous $\mathbb{Z}_p$-algebra 
homomorphisms that induce the identity map on $\mathbb{F}_p$. 

Suppose $\Gamma$ is a finite group and $V$ is a finitely generated $\mathbb{F}_p\Gamma$-module. 
A lift of $V$ over an object $R$ in $\hat{\mathcal{C}}$ is a pair $(M,\phi)$ where $M$ is a finitely generated 
$R\Gamma$-module that is free over $R$, and $\phi:\mathbb{F}_p\otimes_R M\to V$ is an isomorphism of 
$\mathbb{F}_p\Gamma$-modules. Two lifts $(M,\phi)$ and $(M',\phi')$ of $V$ over $R$ are isomorphic if there is an 
isomorphism $\alpha:M\to M'$ with $\phi=\phi'\circ (\mathrm{id}_{\mathbb{F}_p}\otimes\alpha)$. The 
isomorphism class $[M,\phi]$ of a lift $(M,\phi)$ of $V$ over $R$ is called a deformation of $V$ over $R$, and 
the set of such deformations is denoted by $\mathrm{Def}_\Gamma(V,R)$. The deformation functor
$$\hat{F}_V:\hat{\mathcal{C}} \to \mathrm{Sets}$$ 
sends an object $R$ in $\hat{\mathcal{C}}$ to $\mathrm{Def}_\Gamma(V,R)$ and a morphism $f:R\to R'$ in 
$\hat{\mathcal{C}}$ to the map $\mathrm{Def}_\Gamma(V,R) \to \mathrm{Def}_\Gamma(V,R')$ defined by 
$[M,\phi]\mapsto [R'\otimes_{R,f} M,\phi']$, where  $\phi'=\phi$ after identifying 
$\mathbb{F}_p\otimes_{R'}(R'\otimes_{R,f} M)$ with $\mathbb{F}_p\otimes_R M$.

If there exists an object $R(\Gamma,V)$ in $\hat{\mathcal{C}}$ and a deformation $[U(\Gamma,V),\phi_U]$ of 
$V$ over $R(\Gamma,V)$ such that for each $R$ in $\hat{\mathcal{C}}$ and for each lift $(M,\phi)$ of $V$ over 
$R$ there is a unique morphism $\alpha:R(\Gamma,V)\to R$ in $\hat{\mathcal{C}}$ such that 
$\hat{F}_V(\alpha)([U(\Gamma,V),\phi_U])=[M,\phi]$, then we call $R(\Gamma,V)$ the universal deformation 
ring of $V$ and $[U(\Gamma,V),\phi_U]$  the universal deformation of $V$. In other words, 
$R(\Gamma,V)$ represents the functor $\hat{F}_V$ in the sense that $\hat{F}_V$ is naturally isomorphic to 
$\mathrm{Hom}_{\hat{\mathcal{C}}}(R(\Gamma,V),-)$. In the case when the morphism 
$\alpha:R(\Gamma,V)\to R$ relative to the lift $(M,\phi)$ of $V$ over $R$ is only known to be unique if 
$R$ is the ring of dual numbers over $\mathbb{F}_p$ but may be not unique for other $R$, 
$R(\Gamma,V)$ is called the versal deformation ring of $V$ and 
$[U(\Gamma,V),\phi_U]$ is called the versal deformation of $V$.

By \cite{mazur}, every finitely generated $k\Gamma$-module $V$ has a versal deformation ring $R(\Gamma,V)$. 
Moreover, if $V$ is an absolutely irreducible $\mathbb{F}_p\Gamma$-module, then $R(\Gamma,V)$ is
universal.

The following result shows the connection between $R(\Gamma,V)$ and certain first and second cohomology
groups of $\Gamma$ that are related to $V$.

\begin{theorem} {\rm (\cite[\S1.6]{mazur}, \cite[Thm. 2.4]{bockle})}
\label{thm:udr}
Suppose $V$ is an absolutely irreducible $\mathbb{F}_p\Gamma$-module, and let
$d^i_V=\mathrm{dim}_{\mathbb{F}_p}\mathrm{H}^i(\Gamma,\mathrm{Hom}_{\mathbb{F}_p}(V,V))$ 
for $i=1,2$. Then $R(\Gamma,V)$ is isomorphic to 
a quotient algebra $\mathbb{Z}_p[[t_1,\ldots,t_r]]/J$ where $r=d^1_V$ and $d^2_V$ is an upper bound 
on the minimal number of generators of $J$.
\end{theorem}

\section{Cohomology}
\label{s:coh}
\noindent
Let $p$ be an odd prime, and consider a short exact sequence of groups $$0\rightarrow\\N\rightarrow\Gamma\rightarrow G\cong \Gamma/N\rightarrow 1$$ where $N$ is an elementary abelian $p$-group of rank $\ell \geq 2$ and $G$ is a finite group.  We identify $G$ with $\Gamma/N$ in the following.  Note that the action of $G$ = $\Gamma/N$ on $N$ corresponds to an $ \mathbb{F}_{p}$-representation of $G$ denoted by $\phi$.  Let $V$ be a projective $\mathbb{F}_{p}G$-module, and view $V$ also as an ${\mathbb{F}_{p}}\Gamma$-module by inflation.  Let $\tilde{\phi}$ be the contragredient of $\phi$ (i.e. $\tilde{\phi}$ is the dual representation of $\phi$).  Let $V_{\tilde{\phi}}$ (resp. $V_{\tilde{\phi} \wedge \tilde{\phi}}$) denote the $\mathbb{F}_{p}\Gamma$-module associated to $\tilde{\phi}$ (resp. $\tilde{\phi} \wedge \tilde{\phi}$).  If $X$ is a $\Gamma/N$-module, let $X^{\Gamma/N}$ denote the fixed points of the action of $\Gamma/N$.  Let $\otimes$ stand for the tensor product over $\mathbb{F}_p$.  We prove the following result.
\vspace*{.05 in}

\begin{theorem}
\label{th:no1}
Using the above notation,
$$\charg{2}{\Gamma}{{\rm {Hom}}_{\mathbb{F}_{p}}(V,V)}\cong[(V_{\tilde{\phi}}\otimes V^{*}\otimes V)\oplus (V_{\tilde{\phi} \wedge \tilde{\phi}}\otimes V^{*}\otimes V)]^{\Gamma/N}.$$

\noindent
If $N$ is elementary abelian of rank two, the representation $\tilde{\phi} \wedge \tilde{\phi}$ is the one-dimensional representation $\rm{det}\circ (\tilde{\phi}).$
\end{theorem}

In the case when ${\mathbb{F}_{p}}G$ is semisimple, this result will provide a way of using character theory to compute the first and second cohomology groups of $\Gamma$ with coefficients in $\text{Hom}_{\mathbb{F}_{p}}(V,V)$.  To prove Theorem \ref{th:no1} we need the following result.

\vspace*{.05 in}
\noindent
\begin{proposition}
\label{pr:no1}
Let $A$ be a projective ${\mathbb{F}_{p}}G $-module.  Then for all $i \geq 1$, $$A \otimes {\rm H}^{i}(N,{\mathbb{F}}_{p}) \cong {\rm H}^{i}(N,A)$$ as $\mathbb{F}_{p}G$-modules, and $$\charg{i}{\Gamma}{A} \cong \charg{0}{\Gamma/N}{\charg{i}{N}{A}} \cong [\charg{i}{N}{A}]^{G}.$$
\end{proposition}
\noindent
\begin{proof}
Let $i \geq 1$.  We first show that $\charg{i}{N}{A} \cong A \otimes \charg{i}{N}{{\mathbb{F}}_p}$ as ${\mathbb{F}_{p}}G $-modules, where we identify $G$ with $\Gamma/N$ as before.  Let $Z^{i}(N,A)$ denote the space of $i$-cocycles of $N$ with coefficients in $A$, and let $B^i(N,A)$ denote the space of $i$-coboundaries of $N$ with coefficients in $A$.  Let \{$e_{j}$\} be an ${\mathbb{F}}_{p}$-basis for $A$.  Recall, $N$ acts trivially on $A$.

Consider the maps 
\vspace*{.1 in}

\noindent
$\Phi: A \otimes Z^{i}(N,{\mathbb{F}}_{p}) \rightarrow Z^{i}(N,A)$, \hspace*{.3 in} $a\otimes c \xrightarrow{\Phi}$  ${\Delta}_{c,a}$, for all $(a,c)$ in $A \times Z^{i}(N,{\mathbb{F}}_{p})$

\vspace*{.1 in} 

\noindent
$\Psi: Z^{i}(N,A) \rightarrow A \otimes Z^{i}(N,{\mathbb{F}}_{p})$, \hspace*{.3 in} $d \xrightarrow{\Psi} \sum\limits_{j}^{} e_{j}\otimes({e_{j}}^{*}\circ d)$, for all $d$ in $Z^{i}(N,A)$
\vspace*{.1 in}

\noindent
where ${\Delta}_{c,a}$($n_{1},n_{2},...,n_{i}$) = $c$($n_{1},n_{2},...,n_{i}$)$a$ and ${e_{j}}^{*}$ is the dual basis element to $e_{j}$.  Then $\Psi$ and $\Phi$ are ${\mathbb{F}_{p}}G $-module homomorphisms that are inverses of each other which restrict to isomorphisms between $B^i(N,A)$, and $A \otimes B^i(N,{\mathbb{F}}_{p})$.  Thus, ${\rm H}^{i}(N,A) \cong \dfrac{A \otimes Z^{i}(N,{\mathbb{F}}_{p})}{A \otimes B^i(N,{\mathbb{F}}_{p})}$ as $\mathbb{F}_pG$-modules.

Tensoring the short exact sequence of ${\mathbb{F}_{p}}G $-modules
$$0\rightarrow\ B^{i}(N,{\mathbb{F}}_{p})\rightarrow Z^{i}(N,{\mathbb{F}}_{p})\rightarrow {\rm H}^{i}(N,{\mathbb{F}}_{p})\rightarrow 0$$
with $A$ over $\mathbb{F}_{p}$, we obtain $A \otimes {\rm H}^{i}(N,{\mathbb{F}}_{p}) \cong \dfrac{A \otimes Z^{i}(N,{\mathbb{F}}_{p})}{A \otimes B^{i}(N,{\mathbb{F}}_{p})}$ as ${\mathbb{F}_{p}}G $-modules.  Therefore, $A \otimes {\rm H}^{i}(N,{\mathbb{F}}_{p}) \cong {\rm H}^{i}(N,A)$ as $\mathbb{F}_{p}G$-modules, which implies that, in particular, ${\rm H}^{i}(N,A)$ is a projective $\mathbb{F}_{p}G$-module.  

Next, consider the Lyndon-Hochschild-Serre spectral sequence $$\textrm{H}^{p_{0}}(\Gamma/N,\textrm{H}^{q_{0}}(N,A))\Rightarrow\textrm{H}^{p_{0}+q_{0}}(\Gamma,A).$$  Since ${\rm H}^{q_{0}}(N,A)$ is a projective $\mathbb{F}_{p}G$-module for all $q_{0} \geq 1$ by the above argument, and since ${\rm H}^{0}(N,A) \cong A^{N} \cong A$ which is also projective, the terms corresponding to $(p_{0},q_{0}) = (1,i-1), (2,i-2),...,(i,0)$ vanish for i=$p_{0} + q_{0} \geq 1$.  Therefore, ${\rm H}^i(\Gamma,A) \cong {\rm H}^0(\Gamma/N,H^i(N,A))$.
\end{proof}

\noindent
We are now ready to show the main result of the section.  
\vspace*{.05 in}

\noindent
\textit{Proof of Theorem 3.1}.  Recall that $V$ is assumed to be a projective $\mathbb{F}_pG$-module, where we identify $G$ with $\Gamma/N$.  By Proposition \ref{pr:no1}, $\charg{2}{N}{\textrm{Hom}_{\mathbb{F}_{p}}(V,V)}\cong \textrm{Hom}_{\mathbb{F}_{p}}(V,V)\otimes \charg{2}{N}{\mathbb{F}_p}$ as $\mathbb{F}_{p}{\Gamma/N}$-modules.

Consider the Kummer sequence $ 1\rightarrow \mu_{p} \xrightarrow{\iota} \mathbb{C}^{*}\xrightarrow{p} \mathbb{C}^{*}\rightarrow 1$, where $\mathbb{C}^{*}\xrightarrow{p} \mathbb{C}^{*}$ denotes the map given by $z \xrightarrow{p} z^p$. We consider this sequence as a sequence of $\mathbb{Z}N$-modules with trivial $N$-action.

Applying the functor $\textrm{Hom}_{\mathbb{Z}N}(\mathbb{Z},-)$ we obtain the long exact sequence

...$\xrightarrow{\delta} \charg{1}{N}{\mu_p}\xrightarrow{\iota_{*}} \charg{1}{N}{\mathbb{C^{*}}}\xrightarrow{p_{*}}\charg{1}{N}{\mathbb{C^{*}}}\xrightarrow{\delta} \charg{2}{N}{\mu_p} \xrightarrow{\iota_{*}}$
\vspace*{.1 in}

\hspace*{1.1 in}
$\charg{2}{N}{\mathbb{C^{*}}} \xrightarrow{p_{*}} \charg{2}{N}{\mathbb{C^{*}}} \xrightarrow{\delta} \charg{3}{N}{\mu_p} \xrightarrow{\iota_{*}} $...
\vspace*{.1 in}

\noindent
Since $N$ is elementary abelian, $\charg{i}{N}{\mathbb{C^{*}}} \xrightarrow{p_{*}} \charg{i}{N}{\mathbb{C^{*}}}$ is trivial, for $i \geq 1$.  Identifying $\mathbb{F}_p = \mu_p$, we get a short exact sequence of $\mathbb{F}_{p}{\Gamma/N}$-modules 

\hspace*{.5 in}
$0\rightarrow \charg{1}{N}{\mathbb{C^{*}}}\xrightarrow{\delta} \charg{2}{N}{\mathbb{F}_p} \xrightarrow{\iota_{*}} \charg{2}{N}{\mathbb{C^{*}}} \rightarrow 0$.
\vspace*{.1 in}

\noindent
Applying the functor $\textrm{Hom}_{\mathbb{F}_{p}}(V,V)\otimes -$   , and taking fixed points, we obtain, using Proposition \ref{pr:no1},

\vspace*{.1 in}
\noindent
$\charg{2}{\Gamma}{\textrm{Hom}_{\mathbb{F}_{p}}(V,V))\cong  [(\charg{1}{N}{\mathbb{C^{*}}}\otimes {\textrm{Hom}_{\mathbb{F}_{p}}(V,V)}}]^{\Gamma/N}\oplus [\charg{2}{N}{\mathbb{C^{*}}}\otimes {\textrm{Hom}}_{\mathbb{F}_{p}}(V,V)]^{\Gamma/N}$.

\vspace*{.1 in}
\noindent
Therefore, Theorem \ref{th:no1} follows once we show that $\charg{1}{N}{\mathbb{C^{*}}}\cong V_{\tilde{\phi}}$ and $\charg{2}{N}{\mathbb{C^{*}}} \cong V_{\tilde{\phi} \wedge \tilde{\phi}}$ as ${\mathbb{F}_{p}\Gamma/N}$- modules.

Since $N$ is an elementary abelian $p$-group which acts trivially on $\mathbb{C^{*}}, \charg{1}{N}{\mathbb{C^{*}}}=\textrm{Hom}(N,\mathbb{C^{*}})\cong \textrm{Hom}_{\mathbb{F}_{p}}(N,\mathbb{F}_{p})$ as $\mathbb{F}_{p}G$-modules, which implies $\charg{1}{N}{\mathbb{C}^*} \cong V_{\tilde{\phi}}$.  It remains to determine the $\Gamma/N$-module structure of $\charg{2}{N}{\mathbb{C^{*}}}$.  Our result follows after a quick computation, using that $\charg{2}{N}{\mathbb{C^{*}}} \cong N \wedge N$.  This completes the proof of Theorem 3.1.
\bigskip

\noindent
As a consequence of the proof of Theorem 3.1 we obtain the following result.
\begin{corollary}
\label{co:no1}
Under the general hypothesis of Theorem 3.1, we obtain:
\begin{enumerate}

\item[a.]
$\charg{1}{\Gamma}{{\rm Hom}_{\mathbb{F}_{p}}(V,V)} \cong (V_{\tilde{\phi}}\otimes V^{*}\otimes V)^{\Gamma/N} $.

\item[b.]
$\charg{1}{\Gamma}{{\rm Hom}_{\mathbb{F}_{p}}(V,V)}$ is a summand of $\charg{2}{\Gamma}{{\rm Hom}_{\mathbb{F}_{p}}(V,V)}$.

\item[c.]
${\rm dim}_{\mathbb{F}_{p}}(\charg{1}{\Gamma}{{\rm Hom}_{\mathbb{F}_{p}}(V,V))} \leq {\rm dim}_{\mathbb{F}_{p}}(\charg{2}{\Gamma}{{\rm Hom}_{\mathbb{F}_{p}}(V,V))}$.
\end{enumerate}
\end{corollary}
\bigskip

\noindent
For the remainder of the paper, we consider the special case 
$$0\rightarrow\\N\rightarrow\Gamma\rightarrow G = \Gamma/N\rightarrow 1$$ where $\mathbb{F}_{p}G$ is semisimple, $\mathbb{F}_{p}$ is sufficiently large for $G$, and $V$ is an irreducible $\mathbb{F}_{p}G$-module. As before, let $\phi$ denote the action of $G$ on $N$.  
\begin{corollary}
\label{co:center}
Assume the notation of the previous paragraph, and let $\phi$ be irreducible.  Suppose there exists an absolutely irreducible $\mathbb{F}_{p}\Gamma$-module $V_0$ with universal deformation ring $R(\Gamma,V_0) \ncong {\mathbb{Z}}_p$.  Then, the restriction of $\phi$ to the center of $G$ is trivial.
\end{corollary}
\begin{proof}
Let $V_0$ be as in the statement of the corollary.  By Theorem 2.1, $$R(\Gamma,V_0) \cong \mathbb{Z}_p[[t_1,\ldots,t_r]]/J$$ where $r = d^1_{V_0}$, and $d^2_{V_0}$ is an upper bound on the minimal number of generators for $J$.  Since $V_0$ is a projective $\mathbb{F}_{p}G$-module, it has a lift over $\mathbb{Z}_p$.  Because $R(\Gamma,V_0) \ncong \mathbb{Z}_p$, it follows that $d^1_{V_0} \geq 1$.  We now use Corollary \ref{co:no1} for $V = V_0$.  Since we assume $\mathbb{F}_{p}G$ is semisimple, the $\mathbb{F}_{p}$-dimension of the $G$-fixed points of any $\mathbb{F}_{p}G$-module is the multiplicity of the trivial simple $\mathbb{F}_{p}G$-module as a summand.  Recall that we identify $G = \Gamma/N$.  By Corollary \ref{co:no1}, $d^1_{V_0} \geq 1$ implies that $V_{\phi}$ occurs as a summand of the module $V_0^* \otimes V_0$ with the adjoint $G$-action.  Since $V_0$ is absolutely irreducible, the action of an element $z \in Z(G)$ on $V_0^* \otimes V_0$ is given by conjugation with a scalar matrix.  Hence $z$ acts trivially on $V_{\phi}$.
\end{proof}
Our main goal is to relate the universal deformation rings $R(\Gamma,V)$ and the cohomology groups $\charg{i}{\Gamma}{\textrm{Hom}_{\mathbb{F}_{p}}(V,V)}$ for $i$ = 1,2, to the fusion of $N$ in $\Gamma$.
\vspace*{.1 in}

\noindent
We need the following definitions.

\begin{definition}
\label{df:no1}
Let $N, \Gamma, G, \phi$ be as above.
\begin{enumerate}
\item[a.]
For every irreducible $\mathbb{F}_{p}G$-module $V$, let $d_{V}^i$ = ${\rm dim}_{\mathbb{F}_{p}}(\charg{i}{\Gamma}{{\rm Hom}_{\mathbb{F}_{p}}(V,V)}$ for i=1,2.  Note that this number depends on $\phi$.  We say an irreducible $\mathbb{F}_{p}G$-module $V_0$ is cohomologically maximal for $\phi$ if $d_{V_0}^2$ is maximal among all $d_V^2$.  We say an irreducible representation $\rho$ of $G$ over $\mathbb{F}_p$ is cohomologically maximal for $\phi$ if $\rho$ corresponds to an $\mathbb{F}_{p}G$-module with this property.

\item[b.]
We call the orbits of the action $\phi$ of $G$ on $N$ the fusion orbits of $\phi$.  For all $m\geq 1$, let $F_{\phi,m}$ be the number of fusion orbits of $\phi$ with cardinality $m$.  Then, the sequence $\{F_{\phi,m}\}_{m\geq 1}$ is called the fusion numbers of $\phi$.  
\end{enumerate}
\end{definition}

Note that the fusion of $N \textrm{ in } \Gamma$ is uniquely determined by the fusion orbits of $\phi$, since two elements in $N$ are conjugate in $\Gamma$ if and only if they lie in the same fusion orbit of $\phi$.  
\section{Dihedral Groups}
\label{s:dih}
\subsection{Main Results}
\label{ss:main}
In this section, we consider the case when $\ell = 2, n \geq 3$ and $\Gamma/N$ = $G$ is the dihedral group $D_{2n}$ of order $2n$.  That is, we have a short exact sequence of groups
$$0\rightarrow\\N\rightarrow\Gamma\rightarrow G = \Gamma/N\rightarrow 1$$
where $G$ is dihedral and $N$ is an elementary abelian $p$-group of rank two.  Moreover, we assume $\mathbb{F}_{p}G$ is semisimple and $\mathbb{F}_p$ is sufficiently large for $G$.  Again, we let $\phi$ denote the action of $G$ on $N$; and we assume $\phi$ is irreducible.  Our main results, Theorems \ref{th:no2} and \ref{th:no3}, show how the first and second cohomology groups, respectively the universal deformation rings, associated to certain ${\mathbb{F}_p}\Gamma$-modules $V$, can detect the fusion of $N$ in $\Gamma$, i.e. the fusion of $\phi$.  In particular, we will prove Theorem 1.1.

Since $N$ is a $p$-group, every irreducible $\mathbb{F}_p\Gamma$-module is inflated from an irreducible $\mathbb{F}_pG$-module.  Let $\textrm{Rep}_2(G)$ be a complete set of representatives of isomorphism classes of all 2-dimensional representations of $G$ over $\mathbb{F}_p$.  Let $\textrm{Irr}_2(G) \subset \textrm{Rep}_2(G)$ be the subset of isomorphism classes of irreducible 2-dimensional representations.  For $\rho$ in $\textrm{Irr}_2(G)$, let $V_\rho$ be an irreducible ${\mathbb{F}}_{p}G$-module with representation $\rho$. 

We let $n \geq 3$, and consider the standard presentation for $G$ = $D_{2n}$, given by $\langle r,s | r^n, s^2, srs^{-1}r \rangle.$  Moreover, we assume $p \equiv 1 (\textrm{mod } n)$.  Recall that all isomorphism classes of 2-dimensional irreducible representations of $G$ over $\mathbb{F}_p$ are represented by:

$$r\xrightarrow{\theta_{i}} \begin{pmatrix} \omega^{i}&0\\ 0&\omega^{-i} \end{pmatrix} \hspace*{.2 in}
s\xrightarrow{\theta_{i}} \begin{pmatrix} 0&1\\   1&0 \end{pmatrix}$$

\noindent
for $1 \leq i < \frac{n}{2}$, and $\omega$ a primitve $n$th root of unity in $\mathbb{F}_p^*$.  Note that $\theta_i = \mathrm{Ind}_{\langle r \rangle}^G(\chi_i)$, where $\chi_i$ is the one-dimensional representation of $\langle r \rangle$ with $\chi_i(r) = \omega^i$.  For our discussion on dihedral groups $G$, we fix the basis corresponding to the matrices above.

\vspace*{.1 in}


\noindent
\begin{definition}
\label{df:no2}
Define the set map $T: \mathrm{Irr}_{2}(G) \rightarrow \mathrm{Rep}_{2}(G)$ by $T(\theta_i) = T(\mathrm{Ind}_{\langle r \rangle}^G(\chi_i))$ = $\mathrm{Ind}_{\langle r \rangle}^G({\chi}_{i}^2)$.  

If $n$ is odd, let $\Omega$ = $\mathrm{Irr}_2(G)$ = $T(\mathrm{Irr}_2(G))$.  If $n$ is even, let $\Omega$ = $\mathrm{Irr}_2(G) \cap T(\mathrm{Irr}_2(G))$.  In the latter case, for all $\psi$ in $\Omega$, $\mid T^{-1}(\psi)\mid$ = 2.  

Note that $\Omega$ consists precisely of those representations in $\mathrm{Irr}_2(G)$ whose restriction to the center of $G$ is trivial.
\end{definition} 

\begin{theorem}
\label{th:no2}
If $\phi \in \Omega$, then the fusion of $\phi$ is uniquely determined by the set $\{ker(\rho) : \rho \in \mathrm{Irr}_2(G)$ is cohomologically maximal for $\phi \}$ = $\{ker(\rho) : \rho \in \mathrm{Irr}_2(G)$ with $R(\Gamma,V_{\rho}) \ncong \mathbb{Z}_p\}$.
\end{theorem}


\begin{theorem}
\label{th:no3}
Let $G$ = $D_{2n}$.  Let $\textrm{T}$ and $\Omega$ be as above.
\begin{enumerate}
\item[a.]
Let $n$ be arbitrary and let $\phi$ be in $\Omega$.  Then, for any $\psi$ in $\mathrm{Irr}_2(G)$, $\psi$ is cohomologically maximal for $\phi$ if and only if $\textrm{T}(\psi) = \phi$.
\item[b.]
Let $n$ be odd, $\phi_{1}, \phi_{2} \in \mathrm{Irr}_2(G) = \Omega$.  Then $\phi_{1}$ and $\phi_{2}$ have the same fusion if and only if $\textrm{T}^{-1}(\phi_{1}) \textrm{and } T^{-1}(\phi_{2})$ have the same kernel.
\item[c.]
Let $n$ be even, $\phi_{1}, \phi_{2} \in \Omega$.  Then $\phi_{1}$ and $\phi_{2}$ have the same fusion if and only if $\{\textrm{kernel of } \psi: \psi \in \textrm{T}^{-1}(\phi_{1})\}=\{\textrm{kernel of } \psi: \psi \in \textrm{T}^{-1}(\phi_{2})\}$.

\end{enumerate}
\end{theorem}

Theorems \ref{th:no2} and \ref{th:no3} say that for $\phi$ in $\Omega$, the fusion of $N$ in $\Gamma$ can be detected by the cohomology groups, respectively the universal deformation rings, in the following sense.  Given $\phi$ in $\Omega$, we may determine the irreducible representations $\psi$ such that $\psi$ is cohomologically maximal for $\phi$.  Additionally, this assignment is reversible.  That is, given a collection of irreducible representations that are cohomologically maximal for some $\phi$ in $\Omega$, we may determine $\phi$.  Moreover, given only the fusion of $\phi$ in $\Omega$ we can determine the kernels of the representations that are cohomologically maximal for $\phi$.  Analagously, this assignment is again reversible.  In addition, since $\psi$ is cohomologically maximal for $\phi$ in $\Omega$ if and only if $R(\Gamma,V_{\psi}) \ncong \mathbb{Z}_{p}$, the fusion of $N$ in $\Gamma$ can also be determined by the knowledge of the universal deformation rings.

Thus, for $\phi$ in $\Omega$ we have the following one-to-one correspondences:
\vspace*{.1 in}
\begin{align*}
\phi &\leftrightsquigarrow \{\psi \in \mathrm{Irr}_{2}(G): \psi \textrm{ is cohomologically maximal for } \phi\} \\   \phi &\leftrightsquigarrow \{\psi \in \mathrm{Irr}_{2}(G): R(\Gamma,V_{\psi}) \ncong \mathbb{Z}_p \}.
\end{align*}
\begin{align*}
\textrm{ Fusion of } \phi &\leftrightsquigarrow \{\textrm{ker}(\psi) : \psi \in \mathrm{Irr}_{2}(G)\textrm{ is cohomologically maximal for } \phi\} \\  \textrm{ Fusion of } \phi &\leftrightsquigarrow \{\textrm{ker}(\psi) : \psi \in \mathrm{Irr}_{2}(G) \textrm{ and }R(\Gamma,V_{\psi}) \ncong \mathbb{Z}_p \}.
\end{align*}

Theorem 1.1 says that even if $\phi$ is not in $\Omega$, knowledge of all $R(\Gamma,V)$ may still be enough to determine the fusion of $N$ in $\Gamma$.  For a generic choice of $n$, however, $\Omega$ is precisely the set of isomorphism classes of representations for which fusion may be determined.  In subsections 4.2-4.5, we prove our main results.  In subsection 4.6, we briefly discuss the case when $G$ = $\Gamma/N$ is an abelian group and compare this case to the dihedral case.
\subsection{Cohomology for $\bf{D_{2n}}$}
\label{ss:cohfordih}
\noindent
In this subsection we determine $\charg{2}{\Gamma}{\textrm{Hom}_{\mathbb{F}_{p}}(V_{\psi},V_{\psi})}$ for $\phi$ in $\Omega$ and $\psi$ in $\textrm{Irr}_2(G)$.  We make the same assumptions as before.  In particular, $n \geq 3$ and $p \equiv 1 (\textrm{mod }n)$, which means that $\mathbb{F}_pG$ is semisimple and $\mathbb{F}_p$ is sufficiently large for $G$.  Recall, we have that $T: \textrm{Irr}_{2}(G) \rightarrow \textrm{Rep}_{2}(G)$ is given by $T(\theta_i) = T(\mathrm{Ind}_{\langle r \rangle}^G(\chi_i))$ = $\textrm{Ind}_{\langle r \rangle}^G({\chi}_{i}^2)$.  Recall also that for $n$ odd, $T$ is a bijection from $\textrm{Irr}_2(G)$ to $\textrm{Irr}_2(G)$ = $\Omega$.  For $n$ even, $\Omega$ = $\textrm{Irr}_2(G) \cap T(\textrm{Irr}_2(G))$ and $T: \rightarrow \Omega$ is a two to one set map.

\begin{proposition}
\label{pr:no2}
Let $G$ = $D_{2n}$.  Let $\Omega$ and $T$ be as above.  
\begin{enumerate}
\item[a.]
Let $n$ be odd, and let $\phi$ be an element of ${\rm Irr}_2(G)$ = $\Omega$.  Then, there exists a unique $\psi$ = $T^{-1}(\phi) $ in ${\rm Irr}_2(G)$ with $d^2_{V_\psi}$ = 2.  For all other $V$, $d^2_V$ = 1.  So $V_\psi$ is cohomologically maximal for $\phi$.
\item[b.]
Let $n$ be even, and let $\phi$ be an element of $\Omega$.  Then, there exist exactly two $\psi$ in ${\rm Irr}_2(G)$ with $d^2_{V_\psi}$ = 2.  For all other $V, d^2_V$ = 1.  Thus, there are precisely two $\psi$ that are cohomologically maximal for $\phi$.  These representations are exactly the elements of $T^{-1}(\{\phi\})$.
\end{enumerate}
\end{proposition}

\noindent
The proposition follows from the following two lemmas.
\begin{lemma}
\label{le:no1}
Let $G = D_{2n}$, let $1 \leq i < \frac{n}{2}$, let $V = V_{\theta_i}$, and let $\phi = T(\theta_i)$. Then, $$V^{*} \otimes V \cong \mathbb{F}_p \oplus V_{\chi_1} \oplus V_{\phi},$$ as $\mathbb{F}_pG$-modules, where $\mathbb{F}_p$ is the trivial simple $\mathbb{F}_pG$-module and $\chi_1$ is the sign representation.

More precisely, identifying $V^*\otimes V = {\textrm{Hom}}_{\mathbb{F}_p}(V,V) = M_2(\mathbb{F}_p)$ with the adjoint action of $\theta_i$ we obtain:
\begin{enumerate}
\item[a.]
The $\mathbb{F}_p$-span of $\begin{pmatrix} \ 1&0\\ 0&1 \end{pmatrix}$ is isomorphic to the trivial simple $\mathbb{F}_p G$-module  $\mathbb{F}_p$.
\item[b.]
The ${\mathbb{F}}_p$-span of $\begin{pmatrix} \ 1&0\\ 0&-1 \end{pmatrix}$ is isomorphic to $V_{\chi_1}$.
\item[c.]
The ${\mathbb{F}}_p$-span of $f = \begin{pmatrix} \ 0&1\\ 0&0 \end{pmatrix}$ and $g = \begin{pmatrix} \ 0&0\\ 1&0 \end{pmatrix}$ is isomorphic to $V_{\tilde{\phi}}$ which is isomorphic to $V_{\phi}$.
\end{enumerate}
\end{lemma}

\begin{proof}
The first two statements a. and b. are clear.  Statement c. follows since $\theta_i(r) f \theta_i(r)^{-1} = \omega^{2i} f \textrm {, } \theta_i(r) g  \theta_i(r)^{-1} = \omega^{-2i} g, \textrm{ and } \theta_i(s) u \theta_i(s)^{-1} = v, \textrm{ for } \{u, v \} = \{f, g \}$.  
\end{proof}

\begin{lemma}
\label{le:no2}
Let $G = D_{2n}, 1 \leq i,j < \frac{n}{2}$, let $V = V_{{\theta}_i}$, and $\phi = \theta_{j}$.  Then, $d_V^2 = d_V^1 + 1$ and 

$d_V^1$ = $\begin{cases} 0, & \textrm{if } \theta_{j} \neq \textrm{T}(\theta_{i})\\ 1, & \text{if } \theta_{j} = \textrm{T}(\theta_{i}). \end{cases}$
\end{lemma}
\begin{proof}
Define $T(V) = V_{T(\theta_i)}$.  By Lemma \ref{le:no1}, we have $V^{*} \otimes V \cong \mathbb{F}_p \oplus V_{\chi_1} \oplus T(V)$ as $\mathbb{F}_pG$-modules.  Note for any $\phi$ in $\textrm{Irr}_2(G)$, we have $\textrm{det}\circ (\tilde{\phi}) = \chi_{1}$.  Since we assume $\mathbb{F}_pG$ is semisimple, the $\mathbb{F}_p$-dimension of the $G$-fixed points of any $\mathbb{F}_pG$-module is the multiplicity of the trivial simple $\mathbb{F}_pG$-module as a summand.  Recall that we identify $G = \Gamma/N$.  By Theorem \ref{th:no1} and Corollary \ref{co:no1}, we have that $\charg{2}{\Gamma}{\textrm{Hom}_{\mathbb{F}_{p}}(V,V)}\cong[(V_{\tilde{\phi}}\otimes V^{*}\otimes V)\oplus (V_{\rm{det}\circ (\tilde{\phi})}\otimes V^{*}\otimes V)]^{G}$, and $\charg{1}{\Gamma}{\textrm{Hom}_{\mathbb{F}_p}(V,V)}$ = $(V_{\tilde{\phi}}\otimes V^{*}\otimes V)^{G} $.  Hence, $\charg{2}{\Gamma}{\textrm{Hom}_{\mathbb{F}_{p}}(V,V)} \cong [(V_{\tilde{\phi}}\otimes (\mathbb{F}_p \oplus V_{\chi_1} \oplus T(V))]^G \oplus [(V_{\chi_{1}}\otimes (\mathbb{F}_p \oplus V_{\chi_1} \oplus T(V))]^G \cong [(V_{\tilde{\phi}}\otimes (\mathbb{F}_p \oplus V_{\chi_1} \oplus T(V))]^G \oplus [V_{\chi_1} \oplus \mathbb{F}_p \oplus T(V)]^G  \cong [V_{\tilde{\phi}} \oplus V_{\tilde{\phi}} \oplus (V_{\tilde{\phi}} \otimes T(V))]^G \oplus [V_{\chi_1} \oplus \mathbb{F}_p \oplus T(V)]^G$.  It is clear that the trivial simple $\mathbb{F}_pG$-module appears as a summand of the second term with multiplicity 1.  Additionally, the trivial simple $\mathbb{F}_pG$-module is a summand of the first term if and only if $ V_{\phi} \cong V_{\tilde{\phi}} \cong T(V)$, i.e. $\phi = \theta_j = T(\theta_i)$.
\end{proof} 

Observe that we have shown that for all $\phi$ not in $\Omega$, $\charg{2}{\Gamma}{\textrm{Hom}_{\mathbb{F}_{p}}(V,V)}$ is one-dimensional for every two-dimensional irreducible $\mathbb{F}_pG$-module $V$.  Hence, in this case, every $V$ in $\textrm{Irr}_2(G)$ is cohomologically maximal.  By the argument in Corollary \ref{co:center}, we moreover have that for all such $V$, $R(\Gamma,V) \cong {\mathbb{Z}}_p$.  In the following sections, we will show that when this happens, the fusion of $N$ in $\Gamma$ cannot typically be detected by the knowledge of $R(\Gamma,V)$.  For certain choices of $n$, however, the situation is actually better.  More precisely, if $n$ is either a power of $2$, or $n = 2 q$ for some odd prime $q$, then the fusion of $N$ in $\Gamma$ can always be determined by the knowledge of all $R(\Gamma,V)$.  

\subsection{Universal Deformation Rings}
\label{ss:udr}
\noindent
In this subsection we determine the universal deformation ring $R(\Gamma,V)$ for every 2-dimensional irreducible $\mathbb{F}_pG$-module $V$, which we view as an $\mathbb{F}_p\Gamma$-module by inflation.  We continue to assume that $\mathbb{F}_pG$ is semisimple and $\mathbb{F}_p$ is sufficiently large for $G$.  We use a result from \cite[Thm. 3.1]{bleher-chinburg-desmit} to show that if $\charg{2}{\Gamma}{\textrm{Hom}_{\mathbb{F}_{p}}(V,V)}$ is two-dimensional, then $R(\Gamma,V) \cong {\mathbb{Z}}_{p}[[t]]/(t^2,pt)$.  Recall, we have shown that for $G$ = $D_{2n}$, $d^2_V = \textrm{dim}_{{\mathbb{F}}_p}(\charg{2}{\Gamma}{\mathrm{Hom}_{\mathbb{F}_{p}}(V,V))}$ is two-dimensional if and only if $d^1_V$ = 1.  Otherwise $d^2_V$ = 1 and $d^1_V$ = 0.  In the latter case, $R(\Gamma,V)$ is a quotient of ${\mathbb{Z}}_{p}$.  Since any $V$ has a lift to $\mathbb{Z}_p$, it follows that in this case the universal deformation ring is ${\mathbb{Z}}_{p}$.  
\begin{proposition}
\label{pr:no3}
Let $G$ = $D_{2n}$, let $\phi$ be in $\Omega$, and let $V$ be a 2-dimensional irreducible $\mathbb{F}_pG$-module.  Then,
\vspace*{.1 in}

$R(\Gamma,V) = \begin{cases} {\mathbb{Z}}_{p} & \textrm{if V is not cohomologically maximal for } \phi,\\  {\mathbb{Z}}_{p}[[t]]/(t^2,pt) & \textrm{if V is cohomologically maximal for } \phi. \end{cases}$  
\vspace*{.1 in}

\noindent
Additionally, for any $\phi$ in ${\rm Irr}_2(G)$, $R(\Gamma,V) \cong {\mathbb{Z}}_{p}[[t]]/(t^2,pt)$ if and only if $d^2_V$ is equal to two.  Thus, for $\phi$ not in $\Omega$, $R(\Gamma,V) \cong {\mathbb{Z}_p}$.

\end{proposition}

\noindent
\begin{proof}
By our comments before the statement of the proposition, we only need to consider the case when $d_V^2 = 2$.  Following the proof of \cite[Thm. 3.1]{bleher-chinburg-desmit}, let $W$ = ${\mathbb{Z}}_p$ and $R = W[[t]]/(pt, t^2)$.  Since $d^2_V = 2$, it follows from Lemmas \ref{le:no1} and \ref{le:no2} that $V_{\phi}$ is a summand of $V^*\otimes V$.  Identifying $N = \mathbb{F}_p \times \mathbb{F}_p$ and using Lemma \ref{le:no1}, we obtain an injective group homomorphism $\iota: N \rightarrow M_2(\mathbb{F}_p) \cong M_2(W/pW)$ given by $\iota((n_1,n_2)) = n_1f + n_2g = \begin{pmatrix} 0 & n_1\\ n_2 & 0\end{pmatrix}$.  Hence, we have a commutative diagram

$$\begin{tikzpicture}[description/.style={fill=white,inner sep=2pt}]
\matrix (m) [matrix of math nodes, row sep=3em,
column sep=2.5em, text height=1.5ex, text depth=0.25ex]
{ 0 & N & \Gamma & G & 1\\ 
0 & M_{2}(W/pW) & GL_{2}(R) & GL_{2}(W) & 1\\ };
\path[->,font=\scriptsize]
(m-1-1) edge node[auto] {} (m-1-2)
(m-1-2) edge node[auto] {} (m-1-3)
edge node[auto] {$\iota$} (m-2-2)
(m-1-3) edge node[auto] {} (m-1-4)
edge node[auto] {$\rho_{R}$} (m-2-3)
(m-1-4) edge node[auto] {} (m-1-5)
edge node[auto] {$\rho_{W}$} (m-2-4)
(m-2-1) edge node[auto] {} (m-2-2)
(m-2-2) edge node[auto] {$d$} (m-2-3)
(m-2-3) edge node[auto] {} (m-2-4)
(m-2-4) edge node[auto] {} (m-2-5);
\end{tikzpicture}$$
\vspace*{.1 in}

\noindent
where $d$($X$) = $1 + tX$ as in \cite[Thm. 3.1]{bleher-chinburg-desmit}.  We notice that all the arguments in the proof of \cite[Thm. 3.1]{bleher-chinburg-desmit} go through once we have proved that the image under $\iota$ of the group $N$ contains two elements which do not commute with each other under multiplication in $M_{2}(W/pW)$.  Using the notation in Lemma \ref{le:no1}, we see that $f \cdot g \neq g \cdot f$.  Thus, $R(\Gamma,V) \cong {\mathbb{Z}}_{p}[t]/(t^2, pt)$.
\end{proof}



\subsection{Fusion for Dihedral Groups}
\label{ss:fus}

\noindent
In this subsection, we determine the fusion of $\phi \in \textrm{Irr}_2(G)$, which uniquely determines the fusion of $N$ in $\Gamma$ when the action of $G = \Gamma/N$ on $N$ is given by $\phi$ (see Definition 3.5).

\begin{proposition}
Let $G = D_{2n}, 1 \leq i_0 < \frac{n}{2}$ and $\phi$ = $\theta_{{i}_{0}}$.  Let $(i_0,n)$ denote the greatest common divisor of $i_0$ and $n$, and define $k = n/(i_0,n)$.  Let $\omega \in {\mathbb{F}}_p^*$ be a primitive $n$-th root of unity.  Writing each element in $N$ as $\left(\begin{array}{c}x\\y\end{array}\right)$ with respect to the fixed basis for the representation $\theta_{i_0}$, the fusion orbits are as follows:
\begin{enumerate}
\item $ Orb \left(\begin{array}{c}0\\0\end{array}\right) $ = $\left\{ \left(\begin{array}{c}0\\0\end{array}\right) \right\}_.$
\item For $\left(\begin{array}{c}x\\y\end{array}\right) \in \mathbb{F}_{p}^{*} \times \mathbb{F}_{p}^{*}$, $ y/x \in \langle \omega^{i_{0}} \rangle$, we have \\
$ Orb \left(\begin{array}{c}x\\y\end{array}\right) $ = $\left\{ \left(\begin{array}{c}x\\y\end{array}\right)_, \left(\begin{array}{c}{\omega}^{i_0}x\\{\omega}^{-i_0}y\end{array}\right)_{,...,}\left(\begin{array}{c}{\omega}^{(k - 1)i_0}x\\{\omega}^{-(k - 1)i_0}y\end{array}\right)\right\}_.$
\item For $\left(\begin{array}{c}x\\y\end{array}\right)$ not in 1. or 2., we have  \\ $ Orb \left(\begin{array}{c}x\\y\end{array}\right) $ = $\left\{ \left(\begin{array}{c}x\\y\end{array}\right)_, \left(\begin{array}{c}{\omega}^{i_0}x\\{\omega}^{-i_0}y\end{array}\right)_{,...,}\left(\begin{array}{c}{\omega}^{(k - 1)i_0}x\\{\omega}^{-(k - 1)i_0}y\end{array}\right)_,\left(\begin{array}{c}y\\x\end{array}\right)_{,...,}\left(\begin{array}{c}{\omega}^{-(k - 1)i_0}y\\{\omega}^{(k - 1)i_0}x\end{array}\right)\right\}_.$

\end{enumerate}
\end{proposition}

\begin{proof}
As before, $G = \langle r,s \mid r^n, s^2, srs^{-1}r \rangle$.  We have $r^{j} \cdot \left(\begin{array}{c}x\\y\end{array}\right)$= $\left(\begin{array}{c}x\\y\end{array}\right)$ if and only if $\omega^{i_{0}j} \cdot x = x$ and $\omega^{-i_{0}j} \cdot y = y$.  Also, $sr^{j} \cdot \left(\begin{array}{c}x\\y\end{array}\right)$= $\left(\begin{array}{c}x\\y\end{array}\right)$ if and only if $x = \omega^{-i_{0}j} \cdot y$ and $y = \omega^{i_{0}j} \cdot x$. Therefore, for all $\left(\begin{array}{c}x\\y\end{array}\right) \neq \left(\begin{array}{c}0\\0\end{array}\right)$, the intersection of the stabilizer of $\left(\begin{array}{c}x\\y\end{array}\right)$ with $\langle r \rangle$ is $\langle r^{n/(i_0,n)} \rangle = \langle r^k \rangle$.  In fact, for $\left(\begin{array}{c}x\\y\end{array}\right)$ as in $3.$, this is the full stabilizer.  If $\left(\begin{array}{c}x\\y\end{array}\right)$ is as in $2.$, say $y/x = \omega^{i_0 j_0}$, then the full stabilizer is $\langle r^k, sr^{j_0} \rangle$.  This implies that the fusion orbits are as stated in the proposition.
\end{proof}

\begin{corollary}
\noindent
With the same notation as in Proposition 4.8, the fusion of $\phi$ is uniquely determined by the greatest common divisor $(i_0,n)$.  Moreover, the fusion numbers of $\phi$ (see Definition 3.5) uniquely determine the fusion of $\phi$.
\end{corollary}

\begin{proof}
The first statement follows from the stabilizer calculation in the proof of Proposition 4.8.  Moreover, the fusion numbers $F_{\phi,m}$ are as follows (letting $k = n/(i_0,n)$ as before):\\
$F_{\phi,1} = 1$\\
$F_{\phi,k} = p - 1$\\
$F_{\phi,2k} = {\frac{(p - 1)(p + 1 - k)}{2k}}_,$\\
and $F_{\phi,m} = 0$ for all other $m \geq 1$.
\end{proof}
In particular, two representations $\theta_i$, $\theta_{i_0}$ in $\mathrm{Irr}_{2}(G)$ have the same fusion if and only if $(i,n) = (i_0,n)$.
\subsection{Proof of Main Results}
\label{ss:proof}
\noindent
In view of the results proved in subsections \ref{ss:cohfordih}, \ref{ss:udr} and \ref{ss:fus}, to complete the proofs of Theorems \ref{th:no2} and \ref{th:no3}, it remains only to prove the one-to-one correspondence for $\phi \in \Omega$:
\vspace*{.1 in}

\noindent
$$\textrm{Fusion of } \phi \leftrightsquigarrow \{\textrm{ker}(\psi) : \psi \in \mathrm{Irr}_{2}(G)\textrm{ is cohomologically maximal for } \phi\}.$$

\noindent
We note that for any $1 \leq i < \frac{n}{2}$, the kernel of $\theta_i$ is uniquely determined by $(i, n)$.  Moreover, we have $T(\theta_i) = \begin{cases} \theta_{2i} & \textrm{if } 2i < \frac{n}{2}\\ \theta_{n - 2i} & \textrm{otherwise} \end{cases}$  
\vspace*{.1 in}

\noindent
Therefore, for $n$ odd, the result follows since $(i,n) = (i_0,n)$ when $T(\theta_i) = \theta_{i_0}$.  In the case when $n$ is even, let $\theta_{i_0} \in \Omega$, i.e. $1 \leq i_0 \leq \frac{n}{2} - 1$ and $i_0 = 2d_0$ for some $d_0$.  Moreover, $T^{-1}(\theta_{i_0}) = \{ \theta_{d_0}, \theta_{k - {d_0} } \}$ for $k = \frac{n}{2}$.  Therefore, for $n$ even, the result follows from the following lemma.
\begin{lemma}
Let $n$ be even, $k = \frac{n}{2}$, and write $n = 2^{\lambda} \cdot m$, for some odd $m$.  Let $\theta_{i_0} \in \Omega$ and write $i_0 = 2d_0$.  Define $a_0 = (d_0, k)$.  Then $\{(d_{0},n), (k-d_{0},n)\}$  =  $\{(a_{0},n), (k-a_{0},n)\}$.  Moreover, $(i_0,n) = 2a_0, (a_0,n) = a_0$, and $(k - a_0,n) \in \{a_0, 2a_0 \}$  
\end{lemma}

\begin{proof}
Suppose first that $2^{\lambda} \nmid d_0$.  Then $(d_0, n) \mid k$, and hence $(d_0, n) = (d_0, k) = a_0$.  If $2^{\lambda - 1} \nmid d_0$, then $(k - d_0, n) = (k - d_0, k) = a_0 = (k - a_0, k) = (k - a_0, n)$.  If $2^{\lambda - 1} \mid d_0$, but $2^{\lambda} \nmid d_0$, then $k - d_0$ and $k - a_0$ are even, and so $(k - d_0, n) = 2(k - d_0, k) = 2a_0 = 2(k - a_0, k) = (k - a_0, n)$.  On the other hand, if $2^{\lambda} \mid d_0$, then $2^{\lambda} \nmid (k - d_0)$ but $2^{\lambda - 1} \mid (k - d_0)$ and $2^{\lambda - 1} \mid (k - a_0)$.  Hence we can use the above argument to obtain $(d_0, n) = (k - (k - d_0), n) = 2(k - (k - d_0), k) = 2(d_0, k) = 2a_0 = 2(k - a_0, k) = (k - a_0, n)$.
\end{proof}

Thus, Theorems 4.2 and 4.3 are established.  In particular, this proves part a. of Theorem 1.1.  Moreover, we have shown in Corollary 3.4 that for $\phi \notin \Omega$, $R(\Gamma,V) \cong \mathbb{Z}_p$, for all absolutely irreducible $V$.  Therefore, to prove part b. of Theorem 1.1, we consider $D_{2n}$ for $n$ even.  If $n$ is either a power of $2$ or equal to $2 q$ for some odd prime $q$, then $\phi \notin \Omega$ if and only if $\phi$ is faithful.  Thus, if one knows that $R(\Gamma,V) \cong \mathbb{Z}_p$, for all absolutely irreducible $V$, then it must be the case that the fusion of $N$ in $\Gamma$ corresponds to $(1,n)$ in the sense of Corollary 4.9.  On the other hand, if $n$ is even, but not as above, then there must exist some odd prime $v$ such that $\theta_v \notin \Omega$.  But then $\theta_1$ and $\theta_v$ have different fusion, but in both cases $R(\Gamma,V) \cong \mathbb{Z}_p$ for all irreducible $V$.  This, together with Theorems 4.2 and 4.3, completes the proof of Theorem 1.1.


\subsection{Abelian Groups}
\label{ss:ab}
In this subsection, we briefly discuss the case when $\Gamma/N$ = $G$ is an abelian group and compare this case to the dihedral case discussed in subsections 4.1-4.5.  In other words, we consider a short exact sequence of groups
$$0\rightarrow\\N\rightarrow\Gamma\rightarrow G = \Gamma/N\rightarrow 1$$
where $G$ is finite abelian, and $N$ is an elementary abelian $p$-group of rank two.  As before, we assume $\mathbb{F}_{p}G$ is semisimple and $\mathbb{F}_p$ is sufficiently large for $G$.  Let $V$ be an irreducible $\mathbb{F}_{p}G$-module viewed as an $\mathbb{F}_{p}\Gamma$-module via inflation.  Let $\phi$ denote the action of $G$ on $N$.  Since $G$ is abelian, $V$ is one-dimensional, and $\phi$ splits into a direct sum of two one-dimensional representations.  Let $\phi$ = $(\theta_1,\theta_2)$, where $\theta_i : G \rightarrow \mathbb{F}_{p}^*$.  We again analyze the extent to which the universal deformation ring $R(\Gamma,V)$ can see the fusion of $N$ in $\Gamma$.  In contrast to the dihedral case, if $G$ is abelian, then $R(\Gamma,V)$ will only be able to detect some information about fusion.
\vspace*{.1 in}

\noindent
\begin{proposition}
\label{pr:ab0}
Let $G$ be abelian, and let $V$ and $\phi$ be as above.  Let ${\{F_{\phi,m}\}}_{m\geq1}$ be the fusion numbers of $\phi$.
\begin{enumerate}
\item[a.]
The universal deformation ring $R(\Gamma,V) \cong \mathbb{Z}_p$ if and only if $F_{\phi,1} = 1$ if and only if both $\theta_1$ and $\theta_2$ are not trivial if and only if $d^1_V = 0$.
\item[b.]
The universal deformation ring $R(\Gamma,V) \cong \mathbb{Z}_p[\mathbb{Z}/p\mathbb{Z}]$ if and only if $F_{\phi,1} = p$ if and only if exactly one of $\theta_1, \theta_2$ is trivial if and only if $d^1_V = 1$.
\item[c.]
The universal deformation ring $R(\Gamma,V) \cong \mathbb{Z}_p[\mathbb{Z}/p\mathbb{Z} \times \mathbb{Z}/p\mathbb{Z}]$ if and only if $F_{\phi,1} = p^2$ if and only if both $\theta_1, \theta_2$ are trivial if and only if $d^1_V = 2$.
\end{enumerate}
\end{proposition}

In the statement of the proposition, we have added brackets to the group rings for clarity.  The above proposition illustrates the extent to which fusion can be detected by universal deformation rings in the case when $G$ is abelian.  Note that Corollary 3.4 is not applicable, as $\phi$ is reducible.  In contrast to the dihedral case, we get no information by varying $V$, as both $R(\Gamma,V)$ and $d^i_V$ for $i = 1,2$ are constant with respect to $V$.  In the abelian case, while some information about the fusion of $N$ in $\Gamma$ may be detected by the universal deformation ring (and indeed the cohomology), it is simply too coarse to completely determine the full fusion (compare with Theorems 1.1, \ref{th:no2}, and \ref{th:no3}).  Instead, for any absolutely irreducible $V$, $R(\Gamma,V)$ sees only the number of fusion orbits of size $1$, i.e. those elements of $N$ which are not fused.  Additionally, unlike the dihedral case, the fusion numbers are not enough to determine the fusion.  

\bigskip  

\noindent
\textit{Proof of Proposition \ref{pr:ab0}}.  Let $G$ be abelian, and let $V$ and $\phi$ = $(\theta_1,\theta_2)$ be as above.  We first determine the number of fusion orbits of size $1$, i.e. $F_{\phi,1}$.  Considering the action of $\phi$ on $N = {\mathbb{F}}_p \times {\mathbb{F}}_p$, we see that $F_{\phi,1} = p^j$, where $j$ counts how many of $\theta_1$, $\theta_2$ are trivial.  In particular, the fusion of $N$ in $\Gamma$ depends on more than just $F_{\phi,1}$.

Next, we determine $d^i_V, i = 1, 2$.  By Theorem \ref{th:no1} and Corollary \ref{co:no1}, we need to calculate $(V_{\tilde{\phi}}\otimes V^{*}\otimes V)^G$ and $(V_{\rm{det}\circ (\tilde{\phi})}\otimes V^{*}\otimes V)^G$.  Since $V$ is one-dimensional, $V^{*}\otimes V$ is trivial, thus $d^i_V$ is independent of $V$ for $i = 1,2$.  Since $\phi$ = $(\theta_1,\theta_2)$, $d^1_V$ counts how many of $\theta_1$, $\theta_2$ are trivial.  Also, $d^2_V - d^1_V$ is $1$ if $\theta_2 = {\theta_1}^{-1}$, and is $0$ otherwise.  

Finally, we determine $R(\Gamma,V)$.  Since $G$ is abelian, it follows by \cite[\S1.4]{mazur} that $R(\Gamma,V) = \mathbb{Z}_p[\Gamma^{ab, p}]$, where $\Gamma^{ab, p}$ denotes the maximal abelian $p$-quotient of $\Gamma$.  Since the order of $G$ is relatively prime to $p$, $\Gamma^{ab, p}$ can only be the trivial group, or $\mathbb{Z}/p\mathbb{Z}, \textrm{ or } \mathbb{Z}/p\mathbb{Z} \times \mathbb{Z}/p\mathbb{Z}$.  Since $d := d^1_V$ is minimal such that $R(\Gamma,V)$ is a quotient of $\mathbb{Z}_p[[t_1, t_2, ..., t_d]]$, it follows that:

\begin{enumerate}
\item[a.]
$d^1_V = 0$ if and only if $R(\Gamma,V) = \mathbb{Z}_p$,
\item[b.]
$d^1_V = 1$ if and only if $R(\Gamma,V) = \mathbb{Z}_p[\mathbb{Z}/p\mathbb{Z}]$,
\item[c.]
$d^1_V = 2$ if and only if $R(\Gamma,V) = \mathbb{Z}_p[\mathbb{Z}/p\mathbb{Z} \times \mathbb{Z}/p\mathbb{Z}].$
\end{enumerate}
 
\noindent
This completes the proof of Proposition \ref{pr:ab0}.

\bibliographystyle{amsplain}

\end{document}